\documentclass[11pt]{amsart}
\usepackage{amsmath, amsthm, amssymb, latexsym,url}
\usepackage{hyperref}

\allowdisplaybreaks

\author{Joseph Vandehey}
\thanks{Email: \href{mailto:vandehey@uga.edu}{\nolinkurl{vandehey@uga.edu}}}
\title{Non-trivial matrix actions preserve normality for continued fractions}
\date{\today}

\newtheorem{thm}{Theorem}[section]

\newtheorem{lem}[thm]{Lemma}

\theoremstyle{remark}
\newtheorem{rem}[thm]{Remark}

\begin{document}

\begin{abstract}
A seminal result due to Wall states that if $x$ is normal to a given base $b$ then so is $rx+s$ for any rational numbers $r,s$ with $r\neq 0$. We show that a stronger result is true for normality with respect to the continued fraction expansion. In particular, suppose $a,b,c,d\in \mathbb{Z}$ with $ad-bc\neq 0$. Then if $x$ is continued fraction normal, so is $(ax+b)/(cx+d)$.
\end{abstract}

\maketitle

\section{Introduction}

A number $x\in  [0,1)$ with base $10$ expansion $x=0.a_1a_2a_3\dots$ is said to be normal (to base $10$) if for any finite string $s=[c_1,c_2,\dots,c_k]$ we have that 
\[
\lim_{n\to \infty} \frac{\#\{0\le i \le n: a_{i+j} = c_j, 1\le j \le k\}}{n} = \frac{1}{10^k}.
\]
Although almost all real numbers are normal\footnote{For readability's sake, we will no longer say ``to base $10$" every time. The results mentioned here would follow if we replaced $10$ by any integer base $b\ge 2$.}, we still do not know of a single commonly used mathematical constant, such as $\pi$, $e$, or $\sqrt{2}$, that is normal.

In his Ph.D. thesis under D.H. Lehmer, Donald Dines Wall \cite{Wall} proved a series of results on normal numbers which are now considered classical, elementary facts. Among them, Wall proved that if $x$ is normal then $qx+r$ is normal for any rational numbers $q,r$ with $q\neq 0$. Chang \cite{Chang} appears to have discovered this result independently, while Doty, Lutz, and Nandakumar \cite{DLN} knew of Wall's result and reproved it by a different method. Aistleitner \cite{Aistleitner} has given the only significant extension of Wall's result the author is aware of, showing that if $x$ is normal and $y\in\mathbb{R}$ is a number with almost all of its digits equal to $0$, then $x+ry$ is normal to base $10$ for any rational $r$. (See also \cite[pp. 97]{Bugeaud}.)

Although the definition of normality is easily extended to many other digital systems, questions about which operations preserve normality are still unanswered in most cases. Recently, the author, with Airey and Mance, \cite{VandeheyAireyMance} studied how rational multiplication and addition act for $Q$-Cantor series expansions.

However, in this paper, we shall be interested in normality for continued fraction expansions, which we shall abbreviate as CF-normality and define explicitly in a moment. Mend\`{e}s France first asked the question of which operations preserve CF-normality \cite[pp. 17--18]{Mauduit}. He actually asked a simpler question, namely if non-zero rational multiplication preserves simple normality\footnote{Simple normality asks that all of the one-digit strings appear with the correct limiting frequency.} for continued fractions. Yann Bugeaud extended the question to ask if non-zero rational multiplication preserved CF-normality \cite[Problem 10.56, pp. 222]{Bugeaud}. Part of the difficulty of proving such a result comes from the fact that rational multiplication and addition are operations that are not very well understood for continued fractions. Research on these topics appears to have come almost completely from a computational side (``Given a continued fraction expansion $x$, how can we quickly compute the continued fraction expansion of $qx+r$?"). Notable works include Gosper \cite{Gosper}, Raney \cite{Raney}, and Liardet and Stambul \cite{LS}. On the theoretical side (``If $x$ has a continued fraction expansion with property $Y$, does $qx+r$ have property $Z$?"), the author is unaware of any significant result.

We recall some standard definitions for continued fractions. The (regular) continued fraction expansion of a number $x\in \mathbb{R}$ is given by
\[
x=a_0+\cfrac{1}{a_1+\cfrac{1}{a_2+\dots}}, \text{ with } a_0 \in \mathbb{Z} \text{ and } a_i\in \mathbb{N} \text{ for } i\in \mathbb{N}.
\]
We will denote this expansion by $\langle a_0;a_1,a_2,\dots\rangle$ for typographical simplicity. This expansion is infinite if and only if $x$ is irrational. We will refer to the $n$th digit of the continued fraction expansion of $x$ by $a_n(x)$ or just $a_n$ if the choice of $x$ is clear. The Gauss map $T:[0,1)\to [0,1)$ given by 
\[
Tx = \begin{cases}\frac{1}{x}-\left\lfloor \frac{1}{x} \right\rfloor, & x \neq 0, \\ 0, & x=0,\end{cases}
\]
acts as a forward shift on the continued fraction digits, ignoring $a_0$. The Gauss measure $\mu$ given by
\[
\mu(A) = \int_A \frac{1}{(\log 2) (1+x)} \ dx,
\]
is a probability measure, preserved by $T$, and is ergodic with respect to $T$. Given a string $s=[c_1,c_2,\dots,c_k]$, we define the cylinder set of $s$ to be
\[
C_s = \{ x \in [0,1) : a_i(x)=c_i, 1\le i \le k\},
\]
and we say this cylinder set has rank $k$. We shall also need the usual matrix action on real numbers given by
\[
\left( \begin{array}{cc} \alpha & \beta \\ \gamma & \delta \end{array}\right) x = \frac{\alpha x  + \beta}{\gamma x + \delta}
\]

With these definitions in mind, we say that a point $x\in [0,1)$ is CF-normal if for any string $s$, we have
\begin{equation}\label{eq:cfnormdef}
\lim_{n\to \infty} \frac{\#\{0\le i < n : T^i x \in C_s\}}{n} = \mu(C_s).
\end{equation}
Since $T^i x \in C_s$ if and only if the string $s$ appears in the continued fraction expansion of $x$ starting at the $i+1$st position, the left-hand side of \eqref{eq:cfnormdef} represents the limiting frequency with which $s$ appears in the continued fraction expansion of $x$.
By the pointwise ergodic theorem, almost all numbers $x\in [0,1)$ are CF-normal. We extend the definition of CF-normal to all $x\in \mathbb{R}$ to say $x$ is CF-normal if $x-a_0(x)$ is CF-normal.

Our main result will be the following, which not only answers Bugeaud's question in the affirmative, but shows that any non-trivial linear fractional transformation preserves CF-normality.

\begin{thm}\label{thm:main}
Let $M$ be a $2 \times 2$ matrix with coefficients in $\mathbb{Z}$ and non-zero determinant. Let $x\in\mathbb{R}$ be CF-normal. Then $Mx$ is also CF-normal.
\end{thm}

For a very specific class of matrices $M$, the above theorem is trivial. In particular, if $x=\langle 0;a_1,a_2,\dots\rangle $ is CF-normal, $c_0\in \mathbb{Z}$, and $c_1,c_2,\dots,c_k\in \mathbb{N}$, then
\[
\left( \begin{array}{cc} 1 & c_0 \\ 0 & 1 \end{array}\right) \left( \begin{array}{cc} 0 & 1 \\ 1 & c_1 \end{array}\right) \left( \begin{array}{cc} 0 & 1 \\ 1 & c_2 \end{array}\right)\dots  \left( \begin{array}{cc} 0 & 1 \\ 1 & c_k \end{array}\right)x = \langle c_0;c_1,c_2,\dots,c_k,a_1,a_2,\dots\rangle.
\]
It is easy to see that this preserves normality. However, all the matrices on the left here have determinant $\pm 1$, and thus so does their product. In fact, for any matrix with determinant $\pm 1$, the action of the matrix on an irrational number $x$ will alter the head of the expansion and leave the tail unchanged (see \cite[Theorem 2.37]{BPSZ}), thus preserving CF-normality. 

We emphasize that Theorem \ref{thm:main} works for any matrix that does not have determinant zero, not just those with determinant $\pm1$. The reason why we exclude matrices with determinant zero is that in such a case $Mx$ will always be the same \emph{rational} number, regardless of which $x$ is chosen.

Due to earlier work of Kraaikamp and Nakada \cite{KN} and the author \cite{VandeheyJointNormality}, we know that normality for regular continued fraction expansions---the kind we are studying in this paper---is equivalent to normality for nearest-integer continued fractions and continued fractions with odd partial quotients. Thus Theorem \ref{thm:main} holds for these expansions as well.

\subsection{The idea and outline of the proof}

Let us return to the question of normality to base $10$ and give a glimpse into why Wall's result is true.

Given a number $x$ that is base $10$ normal, how often do we expect to see the digit string $7$ appear in the base $10$ expansion of $2x$? We should see a $7$ appear in the $n$th position of $2x$ whenever we see one of the strings $35$, $36$, $37$, $38$, $39$, $85$, $86$, $87$, $88$, or $89$ appear starting in the $n$th position of $x$. We call these strings \emph{trigger strings} for the string $7$. But since $x$ is normal, each of these strings appears with limiting frequency $1/100$ and there are $10$ of them, so we expect to see $7$ appear with limiting frequency $1/10$. In this case, understanding how often trigger strings occur relies on knowing how the sequence $(10^n x)$ is distributed modulo $1$.

A slightly harder problem: if $x$ is base $10$ normal, then how often do  we expect to see $7$ in the base $10$ expansion of $x/3$? Here, to determine what the $n$th digit of $x/3$ is, we must not only know what the $n$th digit of $x$ is, we must know something about all the first $n-1$ digits. In particular, the $n$th digit of $x/3$ is $7$ if the $n$th digit of $x$ is $1$, $2$, or $3$, and the sum of the first $n-1$ digits of $x$ is $2$ modulo $3$. If one could show that each of these options appeared with limiting frequency $1/30$,\footnote{This is possible using the techniques of \cite{VandeheyJointNormality}.} that would give the desired limiting frequency for the string $7$ in the base $10$ expansion of $x/3$. 

This suggests that to show that division by $3$ preserves normality to base $10$, we want to understand how the pairs
\begin{equation}\label{eq:divide3}
\left( 10^n x \pmod{1},\quad \sum_{i=1}^{n-1} a_i(x) \pmod{3} \right)
\end{equation}
distribute in the set $[0,1) \times \{0,1,2\}$. (Here $a_i(x)$ is referring to the $i$th digit of the base $10$ expansion, not the continued fraction expansion, of $x$.)

We will run into a similar difficulty with continued fractions. What we will want is some way to examine the tail of the expansion of $Mx$, and to do this, we want to be able to push the matrix $M$ through the first $n$ digits of the continued fraction expansion for $x$ like so:
\begin{equation}\label{eq:transitionstates}
M\left( \cfrac{1}{a_1+\cfrac{1}{a_2+\dots+\cfrac{1}{a_n+T^n x }}}\right) = b_0+\cfrac{1}{b_1+\cfrac{1}{b_2+\dots+\cfrac{1}{b_m+M_n\left( T^n x\right) }}} .
\end{equation}
with $M_n$ belonging to some set $\mathcal{M}$. 
Our analogy to \eqref{eq:divide3} will be the sequence $(T^n x, M_n)$ and we want to show this  distributes nicely in the space $[0,1)\times \mathcal{M}$.

As an explicit example, suppose $M= \left( \begin{array}{cc} 1 & 0 \\ 0 & 2 \end{array}\right)$. Then
\[
Mx = \frac{1}{2} x = \frac{1}{2} \frac{1}{a_1+Tx} = \frac{1}{2a_1+ 2(Tx) }.
\]
By reinterpreting this in terms of matrices, we get that
\[
\left( \begin{array}{cc} 1 & 0 \\ 0 & 2 \end{array}\right)\left( \begin{array}{cc} 0 & 1 \\ 1 & a_1 \end{array}\right)  = \left( \begin{array}{cc} 0 & 1 \\ 1 & 2a_1 \end{array}\right) \left( \begin{array}{cc} 2 & 0 \\ 0 & 1 \end{array}\right).
\]
If instead, $M=\left(\begin{array}{cc} 2 & 0 \\ 0 & 1\end{array}\right)$, then there are more possibilities, depending on the value of $a_1(x)$:
\begin{align*}
\left( \begin{array}{cc} 2 & 0 \\ 0 & 1 \end{array}\right)\left( \begin{array}{cc} 0 & 1 \\ 1 & 2n \end{array}\right) &= \left( \begin{array}{cc} 0 & 1 \\ 1 & n\end{array}\right) \left( \begin{array}{cc} 1 & 0 \\ 0 & 2 \end{array}\right),\\
\left( \begin{array}{cc} 2 & 0 \\ 0 & 1 \end{array}\right)\left( \begin{array}{cc} 0 & 1 \\ 1 & 1 \end{array}\right) &= \left( \begin{array}{cc} 0 & 2 \\ 1 & 1 \end{array}\right), \text{ and }\\
\left( \begin{array}{cc} 2 & 0 \\ 0 & 1 \end{array}\right)\left( \begin{array}{cc} 0 & 1 \\ 1 & 2n+1 \end{array}\right) &= \left( \begin{array}{cc} 0 & 1 \\ 1 & n \end{array}\right) \left( \begin{array}{cc} 1 & 1 \\ 0 & 2 \end{array}\right).
\end{align*}

The statement of Lemma \ref{lem:transition} will essentially say that if we started with a ``nice" matrix $M$, then we can in fact choose the matrices $M_n$ in \eqref{eq:transitionstates} to always belong to a particular finite set $\mathcal{M}$; and in Remark \ref{rem:firstsimplification}, we show that it suffices to prove Theorem \ref{thm:main} when $M\in \mathcal{M}$. Theorem \ref{thm:traversing} then says that the sequence $(T^n x, M_n)$ is well-distributed with respect to some measure, provided $M$ starts in a subset of $\mathcal{M}$ with good properties. In Section \ref{sec:buildingdynam}, we again show that it suffices to prove Theorem \ref{thm:main} when $M$ is in such a subset. Finally, in Section \ref{sec:triggerstring}, we will define exactly what it means to be a trigger string, and from here complete the proof of Theorem \ref{thm:main} in Section \ref{sec:proof}.

\subsection{Notations and definitions}
We will denote the continued fraction expansion of $x$ by $\langle a_0;a_1,a_2,\dots\rangle $ and the continued fraction expansion of $Mx$ by $\langle b_0;b_1,b_2,\dots\rangle $. We will commonly use $n$ denote an index for $x$ (so we might talk about the $n$th digit of $x$) while we will use $m$ to denote an index for $Mx$. 

We will often refer to an arbitrary string as $s$ and its digits as $[c_0;c_1,c_2,\dots,c_n]$ or (if $c_0=0$) as $[c_1,c_2,\dots,c_n]$ (in this latter case, we say that $s$ is \emph{proper}). Throughout the paper, we will always consider $c_0\ge -1$ and $c_i\in \mathbb{N}$ for $i\in \mathbb{N}$. The length of a string $s=[c_0;c_1,\dots,c_n]$, denoted by $|s|$, is $n$, regardless of the value of $c_0$. The string $[c]$ with no semi-colon will always denote the string $[0;c]$. When we wish to distinguish between strings of digits in $x$ and strings of digits in $Mx$, we will use $s=[c_0;c_1,c_2,\dots,c_n]$ for strings in $x$ and $t=[d_0;d_1,d_2,\dots,d_m]$ for strings in $Mx$. We will also make reference to a string $r$ that is a substring in the expansion of $Mx$.

We will use $M$ to denote an arbitrary matrix and $\mathcal{M}$ to denote a collection of matrices. We will let 
 \[
A_i = \left( \begin{array}{cc} 1 & i \\ 0 & 1 \end{array} \right) \text{ and } J  = \left( \begin{array}{cc} 0 & 1 \\ 1 & 0 \end{array} \right).\]
 Typically the matrices $A_i$ are denoted by $T_i$ or $T^i$, but we wish to preserve the letter $T$ for the Gauss map. 

To any string $s=[c_0;c_1,c_2,\dots,c_n]$ there is a canonical identification with a matrix
\[
M_s = A_{c_0}JA_{c_1}JA_{c_2}\dots JA_{c_n}= \left( \begin{array}{cc} P & P' \\ Q & Q' \end{array}\right), 
\]
where $P/Q =\langle c_0; c_1,\dots,c_{n-1}\rangle$ and $P'/Q' = \langle c_0;c_1,c_2,\dots,c_n\rangle$, both in lowest terms with positive denominator \cite[Lemma 2.8]{BPSZ}. 
With this in mind, we may define the concatenation of strings $s$ and $s'=[c_0';c_1',\dots,c'_{n'}]$, denoted by $ss'$ or $s.s'$, as the string associated to the matrix 
\[
M_s M_{s'}= A_{c_0}JA_{c_1}JA_{c_2}\dots JA_{c_n}A_{c'_0}JA_{c'_1}JA_{c'_2}\dots JA_{c'_{n'}}
\]
reduced according to the following rules: if $c_n+c'_0>0$ then $A_{c_n}A_{c'_0}=A_{c_n+c'_0}$ and otherwise if $c_n+c'_0=0$ then $A_{c_{n-1}}JA_{c_n}A_{c'_0}JA_{c'_1} = A_{c_{n-1}+c'_1}$. We note that the concatenation $[-1;1].[-1;1]$ equals the empty string.

Similarly we may also define the action of a string on a point by
\begin{equation}\label{eq:appends}
s.x = M_s x = \langle c_0;c_1,c_2,\dots,c_n+x\rangle.
\end{equation}
If $x\in [0,1)$ then this amounts to appending the string $s$ to the start of the continued fraction expansion of $x$. It is clear by \eqref{eq:appends} that if $x\in[0,1)$ is CF-normal then $s.x$ is CF-normal for any string $s$.

We shall make use of standard asymptotic notation as well. We will say that $f(n)=O(g(n))$ if there exists a constant $C$ (called an implicit constant) such that $|f(n)|\le C \cdot g(n)$. We will say that $f(n)\asymp g(n)$ (with implicit constant $C$) if $f(n)=O(g(n))$ and $g(n)=O(f(n))$ (both with implicit constant $C$). If we have two $k\times k$ matrices $K_1, K_2$, then we say that $K_1 \asymp K_2$ (with implicit constant $C$) if $(K_1)_{i,j}\asymp (K_2)_{i,j}$ for $1\le i,j\le k$ (uniformly with implicit constant $C$). We will say $f(n)=o(g(n))$ if $f(n)/g(n)\to 0$ as $n \to \infty$. If a variable appears in a subscript of a big-Oh or little-oh, this denotes that the implicit constant or rate of decay is dependent on this variable.

We will say a vector or matrix is non-negative (or positive) if all its coordinates are non-negative (or positive). We will call a vector a probability vector if it is non-negative and the sum of its coordinates is $1$.

\section{Matrices and resultant strings}\label{sec:matrix}

Given $D\in \mathbb{N}$, let $\mathcal{M}_D$ denote the set of $2\times 2$ matrices $\left( \begin{array}{cc} \alpha & \beta \\ \gamma & \delta\end{array}\right)$ that have determinant $\pm D$ and satisfy one of the following conditions:
\begin{itemize}
\item $\gamma=0$, $\beta \ge 0$, $\alpha,\delta>0$, and $\beta<\delta$. (Type I)
\item $\delta=0$, $\alpha\ge 0$, $\beta,\gamma>0$, and $\alpha<\gamma$. (Type II)
\item $\alpha=0$, $\delta\ge 0$, $\beta, \gamma>0$ and $\delta<\beta$. (Type III)
\item $\beta=0$, $\gamma\ge 0$, $\alpha, \delta>0$, and $\gamma<\alpha$. (Type IV)
\item $\alpha<0$, $\beta,\gamma,\delta>0$, and $|\alpha|<\gamma$. (Type V)
\item $\beta<0$, $\alpha,\gamma,\delta>0$, and $|\beta|<\delta$. (Type VI)
\end{itemize}
It is easy to see from the above conditions paired with determinant requirement that $|\alpha|,|\beta|,|\gamma|,|\delta|\le|D|$, and thus $\mathcal{M}_D$ is a finite set.

\begin{lem}\label{lem:transition}
Given any $M\in \mathcal{M}_D$ and $j\in \mathbb{N}$, there exists a matrix $M'\in \mathcal{M}_D$, and integer $m\ge 0$ and integers $d_0\ge -1$ and $d_1,d_2,\dots,d_m\in \mathbb{N}$ such that
\begin{equation}\label{eq:matrixtransfer}
MJ A_j = A_{d_0}JA_{d_1} JA_{d_2} \dots JA_{d_m} M'.
\end{equation}
Moreover, if $d_0=-1$ then $m\ge 1$. 
\end{lem}

This lemma is similar in statement and purpose to Lemma 6 from Liardet and Stambul \cite{LS}. 

We note that we could have that $m=0$ and $d_0=0$ so that $MJA_j = M'$.

\begin{proof}
We let $M=\left( \begin{array}{cc} \alpha & \beta \\ \gamma & \delta\end{array}\right)$. Then
\[
MJA_j = \left( \begin{array}{cc} \beta & \alpha+\beta j \\ \delta & \gamma+ \delta j \end{array}\right).
\]
We call this new matrix $M_{-1}$ and denote its coefficients by $\alpha_{-1}, \beta_{-1}, \gamma_{-1}, \delta_{-1}$ as appropriate. First, we claim that we can choose $d_0\ge -1$ and create a new matrix 
\[
M_0 = A_{d_0}^{-1}M_{-1}= \left( \begin{array}{cc} \alpha_{-1}- d_0 \gamma_{-1} & \beta_{-1}- d_0 \delta_{-1}\\ \gamma_{-1} & \delta_{-1} \end{array}\right),
\]
with coordinates given by $\alpha_0, \beta_0,\gamma_0,\delta_0$, such that either $M_0\in \mathcal{M}_D$ (in which case the proof is complete with $M_0=M'$) or at least one of the inequalities $|\alpha_0|<\gamma_0$ or $|\beta_0|<\delta_0$ holds and at most one of $\alpha_0$ or $\beta_0$ is non-positive. (Since the determinant is non-zero, we cannot have both $\alpha_0$ and $\beta_0$ be $0$.) We proceed by cases.

\emph{Case 1: $M$ is a Type I or Type IV matrix.}  In this case, all coefficients of $M_{-1}$ are positive, except perhaps for $\alpha_{-1}$ which could equal $0$ if we had that $\beta=0$. Regardless, we may choose 
\begin{equation}\label{eq:d0def}
d_0 = \min \left\{ \left\lfloor \frac{\alpha_{-1}}{\gamma_{-1}}\right\rfloor,  \left\lfloor \frac{\beta_{-1}}{\delta_{-1}}\right\rfloor\right\},
\end{equation}
which will guarantee that either $0 \le \alpha_{-1} - d_0 \gamma_{-1} < \gamma_{-1}$ and $0 \le \beta_{-1} - d_0 \delta_{-1}$ (if the first term is minimal), or else $0 \le \beta_{-1} - d_0 \delta_{-1} < \delta_{-1}$ and $0 \le \alpha_{-1} - d_0 \gamma_{-1}$ (if the second term is minimal).

\emph{Case 2: $M$ is a Type II or Type III matrix.} In this case, all coefficients of $M_{-1}$ are positive, except perhaps $\gamma_{-1}$, which could equal $0$ if we had that $\delta=0$. If $\gamma_{-1}=0$, then we choose $d_0=\lfloor \beta_{-1}/\delta_{-1}\rfloor$, so that $0\le \beta_{-1}-d_0 \delta_{-1} < \delta_{-1}$. (This will be a Type I Matrix, finishing the proof in this case.) Otherwise, if $\gamma_{-1}\neq 0$, then we choose $d_0$ as in \eqref{eq:d0def} to the same results.

\emph{Case 3: $M$ is a Type V matrix.} In this case, all coefficients of $M_{-1}$ are positive, except for $\beta_{-1}$ which could be $0$ or even negative. If $\beta_{-1}>0$, then all coefficients are positive, and we again choose $d_0$ as in \eqref{eq:d0def} to the same results. If $\beta_{-1}=0$, then we can let $d_0=0$ as $|\beta_{-1}-d_0\delta_{-1}|< \delta_{-1}$ is trivially true in this case and $\alpha_{-1}-d_0 \gamma_{-1}$ will be positive. If $\beta_{-1}<0$, then $|\beta_{-1}| < |\alpha|<\gamma < \delta_{-1}$, so $M_{-1}$ is already a Type VI matrix and we are done in this case.

\emph{Case 4: $M$ is a Type VI matrix.} In this case, $\gamma_{-1},\delta_{-1}>0$, $\alpha_{-1}<0$ and $\beta_{-1}$ could be zero, positive, or negative. If $\beta_{-1}$ is positive, then since $|\beta|<\delta$, we have that $|\alpha_{-1}|<\gamma_{-1}$, we have that $M_{-1}$ is a Type V matrix and we are done in this case. So suppose $\beta_{-1}\le 0$. Then since $|\beta|<\delta$ by assumption, we have that $|\beta_{-1}| = |\alpha+\beta j| < |\beta j| < |\delta j | < \delta_{-1}$. Thus we have $\alpha_{-1}<0$, $\beta_{-1}\le 0$ and $|\alpha_{-1}| < \gamma_{-1}$ and $|\beta_{-1}| <\delta_{-1}$. So if we choose $d_0=-1$, we have
\[
M_0 = \left( \begin{array}{cc} \alpha_{-1}+\gamma_{-1} & \beta_{-1}+ \delta_{-1}\\ \gamma_{-1} & \delta_{-1} \end{array} \right).
\]
with all the coefficients positive and $0<\alpha_0 < \gamma_0$, satisfying the desired conditions. (Note that in this case, we do not have that $M_0 \in \mathcal{M}_D$, which, with the work below, will guarantee the final line of the statement of the lemma.)

Now we may assume that at least one of the inequalities $|\alpha_0|<\gamma_0$ or $|\beta_0|<\delta_0$ holds and at most one of $\alpha_0$ or $\beta_0$ is non-positive. 

If $M_0\in \mathcal{M}_D$, then we are done. Note that if $M_0$ has either $\alpha_0$ or $\beta_0$ be negative, then it automatically must be a Type V or Type VI matrix. Therefore, let us suppose that $M_0$ is not in $\mathcal{M}_D$, so that all coordinates will be non-negative, with at most one $0$ at $\alpha_0$ or $\beta_0$. Then we choose 
\[
d_1 = \min \left\{ \left\lfloor \frac{\alpha_{0}}{\gamma_{0}}\right\rfloor,  \left\lfloor \frac{\beta_{0}}{\delta_{0}}\right\rfloor\right\},
\]
so that the matrix $M_1 = A_{d_1}^{-1} J M_{0}$ with coordinates $\alpha_1,\beta_1,\gamma_1,\delta_1$ satisfies $|\alpha_1|<\gamma_1$ or $|\beta_1|<\delta_1$  with at most one of $\alpha_1$ or $\beta_1$ is non-positive. If $M_1$ is not in $\mathcal{M}_D$, then we may choose $d_2\ge 1$ in the same way that we chose $d_1$ and construct the matrix $M_2= A_{d_2}^{-1} J M_1$ satisfying an analogous property, and so on. 



At some point as we iterate this procedure we must reach a matrix $M_m\in \mathcal{M}_D$: at every stage, as we go from the matrix $M_i$ to the matrix $M_{i+1}$, we rearrange the coordinates and subtract at least $1$ from one of them (no coefficients grow in size), and we know that the moment we hit a negative coefficient we will have a Type V or Type VI matrix, which is in $\mathcal{M}_D$. Thus after some finite number of steps we must arrive at a matrix $M_m\in \mathcal{M}_D$.

This completes the proof.\end{proof}

\begin{lem}\label{lem:transitionlength}
In Lemma \ref{lem:transition}, the integer $m$ is uniformly bounded by a constant dependent on $D$.
\end{lem}

\begin{proof}Let \eqref{eq:matrixtransfer} holds for some matrices $M,M'$ and integers $j,d_0,d_1,\dots,d_m$.
Let the coordinates of $M$ and $M'$ be $\alpha,\beta,\gamma,\delta$ and $\alpha',\beta',\gamma',\delta'$ respectively.

 Let $P/Q = \langle d_0;d_1,d_2,\dots,d_{m-1}\rangle$ and $P'/Q' = \langle d_0;d_1,d_2,\dots,d_m\rangle$. We may assume, without loss of generality, that $m\ge 2$, so that both $Q$ and $Q'$ are positive. 

By \eqref{eq:matrixtransfer}, we have for any $x\in \mathbb{R}$ with $x\neq- (\gamma+j\delta)/\delta$ that
\[
MJA_j x = A_{d_0}JA_{d_1}\dots JA_{d_m} M' x = \left(\begin{array}{cc} P & P' \\  Q & Q' \end{array}\right) M' x.
\]
Allowing $x$ to tend to infinity implies that 
\[
\frac{\beta}{\delta} = \left(\begin{array}{cc} P & P' \\  Q & Q' \end{array}\right)  \frac{\alpha'}{\gamma'} = \frac{P\alpha'+P' \gamma'}{Q\alpha' + Q' \gamma'},
\]
where if any denominator equals $0$, we take the corresponding fraction to be equal to infinity. We prove the lemma by cases.

\emph{Case 1: $\delta=0$, $\alpha'=0$}. In this case, we cannot also have that $\gamma'_0 =0$ as $M'$ has a non-zero determinant, so we must have that $Q'=0$. However, we already know that $Q'$ is positive, so this case does not occur.

\emph{Case 2: $\delta=0$, $\alpha'\neq 0$}. In this case, we must have that $Q/Q' =-\gamma'/\alpha'$. But it is well-known that $Q/Q' = \langle 0;d_m,d_{m-1},d_{m-2},\dots,d_1\rangle$ (see \cite[pp. 32]{BPSZ}). There are only finitely many choices for $-\gamma'/\alpha'$, since $\mathcal{M}_D$ is a finite set, so  $m$ must be uniformly bounded in this case. 

\emph{Case 3: $\delta \neq 0$, $\gamma'=0$.} In this case, since $M'$ has a non-zero determinant, we have that $\alpha'$ cannot be $0$. Thus, $\beta/\delta = P/Q = \langle d_0;d_1,d_2,\dots,d_{m-1}\rangle$. Again, there are only finitely many possible choices for the rational number $\beta/\delta$, so $m$ must be uniformly bounded in this case.

\emph{Case 4: $\delta \neq 0$, $\gamma'\neq 0$} Let $\alpha'/\gamma' = \langle d_0'; d_1', \dots, d'_{m'} \rangle$. By the definition of $\mathcal{M}_D$, we must have $d_0' \ge -1$. If $d_m + d_0' \ge 1$, then we have
\[
\frac{\beta}{\delta}=\frac{P\alpha'+P' \gamma'}{Q\alpha' + Q' \gamma'} = \langle d_0;d_1,\dots,d_m+d_0',d_1',d_2',\dots,d'_{m'}\rangle.
\]
There are again only finitely many possible choices for $\beta/\delta$ and for each choice, the continued fraction expansion is completely fixed (up to the last digit). Likewise with $\alpha'/\gamma'$. Thus, for each of these finitely many choices, we must have that the sequence $\{d_0,d_1,\dots,d_m\}$ is fixed, and hence $m$ must be uniformly bounded.

If $d_m+d_0' = 0$, then 
\[
\frac{\beta}{\delta}=\frac{P\alpha'+P' \gamma'}{Q\alpha' + Q' \gamma'} = \langle d_0;d_1,\dots,d_{m-1}+d_1',d_2',\dots,d'_{m'}\rangle.
\]
The same argument applies here.
\end{proof}

By Lemma \ref{lem:transition}, we can construct functions $R([j],M)$, $U([j],M)$ that act on length-one strings of positive digits and matrices $M\in \mathcal{M}_D$ such that $R([j],M)=[d_0;d_1,d_2,\dots,d_m]$ and $U([j],M)=M'$ for a sequence $d_i$, $0\le i \le m$, and $M'$ satisfying \eqref{eq:matrixtransfer}. These functions need not be unique. We can then extend these functions to act on arbitrary proper strings $s$ in their first coordinate. We let 
\begin{equation}\label{eq:Udef}
U([c_1,c_2,\dots,c_n],M) = U\Big([c_2,\dots,c_m],U([c_1],M)\Big)
\end{equation}
and
\begin{equation}\label{eq:Rdef}
R([c_1,c_2,\dots,c_n],M) = R([c_1],M) . R\Big([c_2,c_3,\dots,c_n],U([c_1],M)\Big).
\end{equation}
The string $R(s,M)$ is called the resultant string of $s$ and $M$. With these definitions,  if $x=\langle a_1,a_2,a_3,\dots\rangle $, $M\in \mathcal{M}_D$, then  by applying  \eqref{eq:matrixtransfer} twice, we have that
\begin{align*}
Mx &= M\langle a_1,a_2,\dots,a_n+T^n x\rangle\\
&= M A_{a_1} A_{a_2} \dots A_{a_n} (T^n x)\\
&= R([a_1],M).\Big( U([a_1],M)  A_{a_2} \dots A_{a_n} (T^n x) \Big) \\
&= R([a_1],M).\Big( R\big([a_2],U([a_1],M)\big). \big( U([a_2],U([a_1],M))  A_{a_3} \dots A_{a_n} (T^n x)\big) \Big) \\
&= R([a_1,a_2],M). \Big( U([a_1,a_2],M)  A_{a_3} \dots A_{a_n} (T^n x)\Big) 
\end{align*}
where the last equality follows by  \eqref{eq:Udef} and \eqref{eq:Rdef}. By applying  \eqref{eq:matrixtransfer} another $n-2$ more times, we see that
\begin{equation}\label{eq:criticalrelation}
Mx =R([a_1,a_2,\dots,a_n],M).\Big( U([a_1,a_2,\dots,a_n],M)\left( T^nx\right) \Big).
\end{equation}

As noted earlier, the functions $R$ and $U$ need not be unique, but we will assume throughout the rest of the paper that we are referring to the same function $R$ and $U$ every time.

\begin{lem}\label{lem:contractingterms}
There are only finitely many pairs $([j],M)$, $j\in \mathbb{N}$, $M\in \mathcal{M}_D$, whose resultant string is $[-1;1]$. 

Moreover, suppose we have a pair $([c_1,c_2,\dots,c_k],M)$, $c_i\in \mathbb{N}$, $M\in \mathcal{M}_D$, such that for all $i$ in the interval $1\le i \le k$, we have $R([c_i],U([c_1,c_2,\dots,c_{i-1},M))=[-1;1]$. Then $k$ is uniformly bounded by a constant dependent on $D$.
\end{lem}

\begin{proof}
Suppose, as in Lemma \ref{lem:transition}, that we have $MJA_j = A_{-1}J A_{1} M'$. By rewriting this, we obtain $A_j = J M^{-1} A_{-1}J A_1 M'$. (Note that $M^{-1}$ may not have integer coefficients, but the overall matrix must.) Since $j$ is completely determined by $M$ and $M'$ and there are only finitely many possibilities for the two matrices as both are in $\mathcal{M}_D$, the first part of the lemma is immediate.

For the second part, it suffices to show that $k$ cannot be a large even number. So suppose $k$ is even. Then we have
\[
M JA_{c_1}J A_{c_2}\dots J A_{c_k} = \left( A_{-1}JA_1\right)^{k} M',
\]
but $(A_{-1}JA_1)^2=I$, so we have
\[
M JA_{c_1}J A_{c_2}\dots J A_{c_k} =  M'.
\]
By standard facts about continued fractions, all the coordinates of $JA_{c_1}J A_{c_2}\dots J A_{c_k}$ should have size at least on the order of the $k$th Fibonacci number (see \cite[Lemma 2.9]{BPSZ} and note that all $a_n\ge 1$), and since the bottom row of $M$ consists of at least one positive value and one non-negative value, the coordinates in the bottom row of $M'$ must have size at least on the order of the $k$th Fibonacci number. But since the coordinates of $M'$ are bounded, $k$ must be bounded as well.
\end{proof}

\begin{rem}\label{rem:firstsimplification}

At this point, we claim that to prove Theorem \ref{thm:main} it suffices to show that it holds for all CF-normal $x\in [0,1)$ and $M\in \mathcal{M}_D$ for some $D\in \mathbb{N}$. First, note that by replacing $M$ with $MA_{a_0}$ and $x$ with $x-a_0$, which is CF-normal if and only if $x$ is, we may assume without loss of generality that $x\in [0,1)$.

Now consider a matrix $M= \left( \begin{array}{cc} \alpha & \beta \\ \gamma & \delta \end{array}\right)$  with $|\det(M)|=D\neq 0$ and let $x=\langle 0;a_1,a_2,\dots\rangle$ be a CF-normal number. Let $p_{n}/q_n= \langle a_0;a_1,\dots,a_n\rangle$ be the convergents to $x$. Note that we have \[ A_{a_0}J A_{a_1}JA_{a_2}\dots JA_{a_n} = \left( \begin{array}{cc} p_{n-1} & p_n \\ q_{n-1} &  q_n \end{array}\right).\] Since $\left( \frac{p_n}{q_n} \gamma+\delta\right)$ converges to the irrational number $x\gamma+\delta$, we have that for sufficiently large number $n$, the numbers $p_{n-1} \gamma + q_{n-1} \delta$ and  $p_n \gamma + q_n \delta$ will have the same sign. Let $M_{-1}=M A_{a_0}JA_{a_1}\dots J A_{a_n}$, so that the bottom row of $M_{-1}$ has the same sign and
\[
Mx = M_{-1}\left(   T^{n}x \right)
\]
Without loss of generality, we may assume that the bottom row of $M_{-1}$ is strictly positive (if not we could multiply both sides by $-I$, which does not alter the way the matrix $M$ acts on points $x$). We then let $d_0\in \mathbb{Z}$ be the smallest integer so that $A_{d_0}^{-1} M_{-1}$ has all non-negative coefficients. Now we can apply the Euclidean-type algorithm of the proof of Lemma \ref{lem:transition} to $M_0=A_{d_0}^{-1}M_{-1}$, since this $M_0$ satisfies the same properties as the $M_0$ in the proof of Lemma \ref{lem:transition}. (We note that $d_0$ here could be any integer and not just integers greater than or equal to $-1$; however, this does not impact the rest of the proof.) Therefore $A_{d_0}^{-1}M^{-1}$ can be written as \[ JA_{d_1}JA_{d_2}\dots J A_{d_m} M' \] for some string $[d_1,d_2,\dots,d_m]$ and $M'\in \mathcal{M}_D$. Thus,
\[
M x = [d_0;d_1,d_2,\dots,d_m]. M'\left( T^n  x \right).
\]
But as we have that $x\in[0,1)$ is CF-normal if and only if $T^n x$ is CF-normal, and if and only if $s.x$ is CF-normal for any string $s$, it therefore suffices to show that Theorem \ref{thm:main} holds for CF-normal $x\in [0,1)$ and $M\in \mathcal{M}_D$ for some $D\in \mathbb{N}$.

\end{rem}

\section{Normality on a larger space}\label{sec:tildespace}

Here we want to make our first step towards showing that the sequence \[(T^n x, U([a_1,a_2,\dots,a_n],M))\] is nicely distributed in some sense. This will allow us to show that the last term of \eqref{eq:criticalrelation}, which gives the tail of the CF expansion of $Mx$, is likewise nicely distributed.

Let $\Omega\subset [0,1)$ denote the subset of irrational points and let $\mathcal{M}$ denote some finite set, which we will later take to be a set of matrices. We will let $x$ denote elements in $\Omega$ and $M$ denote elements of $\mathcal{M}$. We will consider cylinder sets of $\Omega$ to be the intersection of the usual cylinder sets (for the continued fraction expansion) of $[0,1)$ with $\Omega$. 

We wish to extend the Gauss map $T$ to a transformation $\widetilde{T}$ on a larger domain $\widetilde{\Omega}=\Omega \times \mathcal{M}$. For any $(x,M)\in \widetilde{\Omega}$, we define \[\widetilde{T}(x,M) = (Tx,f_{a_1(x)} (M))\] for some function $f_{a_1}: \mathcal{M}\to \mathcal{M}$ that is indexed by the first digit of $x$.  Since the second coordinate of $\widetilde{T}(x, M)$ only depends on $M$ and the first CF digit of $x$, we see that this second coordinate is constant for all $x$ in the same rank $1$ cylinder. Given a cylinder set $C_s$ for $\Omega$, we call $C_s \times \{M\}$ (for any $M\in \mathcal{M}$) a cylinder set for $\widetilde{\Omega}$. Moreover, we define $\widetilde{\mu}(E\times \{M\}) =  \mu(E)/|\mathcal{M}|$ for any measurable subset $E$ of $\Omega$ and $b\in \mathcal{M}$. 

For easier readability, we will use $(E,M)$ to denote $E \times \{M\}$ for any measurable set $E\subset \Omega$, with measurability being determined by Lebesgue measure or, equivalently, the Gauss measure. We will also let $(E,\mathcal{M})$ denote $E\times \mathcal{M}$.

We adapt our definition of normality on this space. We will say that $(x,M)\in \widetilde{\Omega}$ is $\widetilde{T}$-normal with respect to a measure $\rho$ on $\widetilde{\Omega}$, if for any cylinder set $(C_s,M')$ we have
\[
\lim_{n\to \infty} \frac{\#\{0\le i < n: \widetilde{T}^i(x,M)\in (C_s,M')\}}{n} = \rho (C_s,M').
\]

We say $\widetilde{T}$ is transitive if for any $M_1,M_2\in \mathcal{M}$, there exists a proper string $s$ of length $n$ such that \[ T^n( C_s,M_1) = (\Omega,M_2). \]

The goal of this section is to prove the following result, which is similar to a previous result of the author \cite{VandeheyJointNormality}; however, as this paper contains significant departures (notably not assuming that the functions $f_a$ are bijective and hence not being able to assume that $\tilde{\mu}$ is $\widetilde{T}$-invariant), we present the proof in full.

\begin{thm}\label{thm:traversing}
If $\widetilde{T}$ is transitive, then there exists a probability measure $\rho$ on $\widetilde{\Omega}$ that is absolutely continuous with respect to $\widetilde{\mu}$ and such that $\widetilde{T}$ preserves $\rho$ and is ergodic with respect to $\rho$.  Moreover, if $x$ is CF-normal, then for any $M\in \mathcal{M}$, the point $(x, M)$ is $\widetilde{T}$-normal with respect to $\rho$.
\end{thm}

Since almost all numbers $x\in [0,1)$ are CF-normal, almost all $(x,M)\in \widetilde{\Omega}$ are $\widetilde{T}$-normal with respect to $\rho$. Therefore, calling these points $(x,M)$ ``normal" is reasonable to do.

\subsection{Necessary Lemmata for Theorem \ref{thm:traversing}}

In order to simplify the readability of the proof of Theorem \ref{thm:traversing}, we will include several lemmas here. All these results make use of the definitions and assumptions at the start of section \ref{sec:tildespace}.

In order to show that $(x,M)$ is $\widetilde{T}$-normal with respect to $\rho$, we will need to make use of the Pyatetski\u\i-Shapiro normality criterion in the following form.

\begin{lem}\label{lem:PS}
Let $(x,M)\in \widetilde{\Omega}$ and suppose a measure $\rho$ exists satisfying the first part of Theorem \ref{thm:traversing}. If for any cylinder set $(C_s,M')$, we have
\[
\limsup_{n\to \infty} \frac{\#\{0\le i \le n-1: T^i (x,M) \in (C_s,M')\}}{n} \le \sigma\cdot \rho(C_s,M').
\]
for some uniform constant $\sigma$, then $x$ is $\widetilde{T}$-normal with respect to $\rho$.
\end{lem}

\begin{proof}
This is a simple consequence of Theorem 1 in \cite{MS}. We briefly describe how using the terminology from their paper.

 We let the family $\{C_m\}$ denote the family of all cylinder sets on $\widetilde{\Omega}$ and also let $\varphi(t)=\sigma\cdot t$. Since the set $A_l(T,\chi_I,\delta)$ is fixed on rank-$l$ cylinder sets, we have that $H_\varphi(A_l(T,\chi_I,\delta))$ equals $\sigma\cdot \mu(A_l(T,\chi_I,\delta))$ and thus goes to $0$ as $l\to \infty$.
\end{proof}

A string $s=[c_1, c_2, \dots , c_n]$ is said to be traversing if for every $M_1$ and $M_2$ there exists a $i<n$ such that
\[
\widetilde{T}^i \left(C_{s} ,M_1\right) \subset (\Omega,M_2) .
\]
The interesting property of the traversing string is that for any $x\in C_{s}$, we have that the $\widetilde{T}$-orbit of $(x, M)$ eventually traverses all of $\mathcal{M}$ in its second coordinate, regardless of which $M$ it started with.

\begin{lem}\label{lem:traversing}
If $\tilde{T}$ is transitive, then a traversing string exists.
\end{lem}

\begin{proof}
Let $M_1,M_2,\dots,M_{|\mathcal{M}|}$ be the elements of $\mathcal{M}$. By the transitivity of $\widetilde{T}$, there exists a finite string $s_1$ such that 
\[
\{M:\text{there exists }i<|s_1|\text{ such that } \widetilde{T}^i(x,M_1)=(T^ix, M)\text{ for all }x\in C_{s_1}\}=\mathcal{M},
\]
that is, that the second coordinates of $\widetilde{T}^i(x,M_1)$, $i\le |s_1|$ traverses all of $\mathcal{M}$. Let $M_2'\in \mathcal{M}$ be given by $T^{|s_1|}(x,M_2) = (*,M_2')$ for any $x\in C_{s_1}$. Again by transitivity, we can find a string $s_2$ such that 
\[
\{M:\text{there exists } i<|s_2|\text{ such that } \widetilde{T}^i(x,M_2')=(T^ix, M)\text{ for all }x\in C_{s_2}\}=\mathcal{M}.
\]
But then by construction, for any $x\in C_{s_1s_2}$, the second coordinates of \emph{both} $\widetilde{T}^i (x;M_1)$ and $\widetilde{T}^i (x;M_2)$, for $i\le |s_1|+|s_2|$, traverse all of $\mathcal{M}$.  In this way we can continue to find strings $s_3,s_4,\dots,s_{|\mathcal{M}|}$, so that the concatenated string $s_1s_2s_3\dots s_{|\mathcal{M}|}$ is the desired traversing string.
\end{proof}

A well-known consequence of Renyi's condition for continued fraction expansions (see Chapter 9 of \cite{SchweigerBook}) states that there exists an absolute constant $\mathcal{C}>0$ so that for any measurable set $E$ and cylinder $C_s$ of rank $n$, we have that
\begin{equation}\label{eq:renyi}
 \frac{1}{\mathcal{C}} \mu(E) \mu(C_s)\le \mu(T^{-n}E \cap C_s ) \le \mathcal{C} \mu(E) \mu(C_s).
\end{equation}
It is clear that one could replace $C_s$ by any set that can be expressed as a union of rank $n$ cylinder sets. 

We will also want a similar equality (sans the cylinder set) to hold for $\widetilde{T}$ and $\tilde{\mu}$, for which we will require the following key result.

\begin{lem}\label{lem:markovchain}
Let $\{K_n\}_{n=1}^\infty$ be a sequence of $k \times k$ Markov matrices such that $K_{n_1}\asymp K_{n_2}$ uniformly for $n_1,n_2 \in \mathbb{N}$. Assume that there exists a power $\ell$ such that $K_1^\ell$ has all positive coordinates, and also assume that $\{K_1\}_{i,i}>0$ for $1\le i \le k$. Then there exists an integer $n_0\in \mathbb{N}$ and a constant $c\in (0,1)$ such that for any $1\times k$ probability vector $\vec{v}$ and any $n\ge n_0$, we have that all coordinates of $\vec{v} K_1K_2 K_3 \dots K_n$ are in the interval $(c,1-c)$.
\end{lem}

\begin{proof}
This is a special case of Proposition 2.13 in \cite{SCZ}: the restriction that $K_{n_1}\asymp K_{n_2}$ uniformly for  $n_1,n_2 \in \mathbb{N}$ implies  that $K_{n+1}K_{n+2}\dots K_{n+\ell}\asymp K_1^\ell$ with an implicit constant dependent on $\ell$. This allows us to replace the ``uniform irreducibility'' assumption with our simpler condition that there exists a power $\ell$ such that $K_1^\ell$ has all positive coordinates.
\end{proof}

\begin{lem}\label{lem:bigrenyi}
Suppose $\widetilde{T}$ is transitive, then there exists a constant $\mathcal{D}>0$ and integer $n_0\in \mathbb{N}$ such that for any measurable set $E\subset \widetilde{\Omega}$ and any $n\ge n_0$, we have that
\begin{equation}\label{eq:bigrenyi}
\frac{1}{\mathcal{D}} \tilde{\mu}(E) \le \tilde{\mu}(\widetilde{T}^{-n} E) \le \mathcal{D}\tilde{\mu}(E).
\end{equation}
\end{lem}

\begin{proof}
We will make a few definitions to start. Let $\mathcal{S}_{n,i}$ denote the set of cylinders $(C_s,M)\subset \widetilde{\Omega}$ such that $|s|=n$ and $\widetilde{T}^n (C_s,M) = (\Omega,M_i)$. We let $\mathcal{A}_{i,j}$ denote the set of $a\in \mathbb{N}$ such that $f_a(M_i)=M_j$. We also consider a sequence of probability vectors $\vec{v}_n= \{v_{n,1}, v_{n,2},\dots,v_{n,|\mathcal{M}|}\}$, $n\ge 0$, such that
\[
v_{n,i} = \tilde{\mu}\left( \bigcup_{(C_s,M)\in \mathcal{S}_{n,i}} (C_s,M) \right) = \tilde{\mu}\left( \widetilde{T}^{-n} (\Omega,M_i) \right).
\]
For example, $\vec{v}_0 = \{1/|\mathcal{M}| , 1/|\mathcal{M}| , \dots, 1/|\mathcal{M}| \}$.

Our first goal will be to show that there exists a constant $c'\in(0,1)$ and integer $n_0'\in \mathbb{N}$ so that $v_{n,i}\in (c',1-c')$ if $n\ge n_0'$.

Consider a sequence of $|\mathcal{M}|\times |\mathcal{M}|$ matrices $(K_n)_{n=1}^\infty$ defined by
\[
(K_n)_{i,j} = \frac{1}{v_{n,i}} \tilde{\mu}\left(  \bigcup_{(C_s,M)\in \mathcal{S}_{n,i}} \bigcup_{a\in \mathcal{A}_{i,j}}  (C_{s.[a]},M)   \right) = \frac{ \tilde{\mu}\left( \widetilde{T}^{-n}(\Omega,M_i) \cap \widetilde{T}^{-n-1}(\Omega,M_j)\right)}{\tilde{\mu}\left( \widetilde{T}^{-n} (\Omega,M_i) \right)}.
\]
It is clear by construction that these matrices are stochastic and that $\vec{v}_{n+1}=\vec{v}_n K_n$.

We let $K_{\ell,n}= K_{\ell+1}K_{\ell+2}\dots K_n$.

Unfortunately, the matrices $K_n$ are not all the same, so $K_{0,n}$ represents a  Markov chain that is time-inhomogeneous. However, they are not far from being time-homogenous. By writing the matrix coefficients in a different way and applying \eqref{eq:renyi}, we see that
\begin{align*}
(K_n)_{i,j}&= \frac{1}{v_{n,i}} \tilde{\mu}\left(  \bigcup_{(C_s,M)\in \mathcal{S}_{n,i}} \bigcup_{a\in \mathcal{A}_{i,j}}  (C_{s.[a]},M)   \right)\\
&= \frac{1}{v_{n,i}} \tilde{\mu}\left(\bigcup_{(C_s,M)\in \mathcal{S}_{n,i}}\left(C_s \cap T^{-n}\left(\bigcup_{a\in \mathcal{A}_{i,j}} C_{[a]}\right),M\right)\right) \\
&\asymp \frac{1}{v_{n,i}}  \tilde{\mu}\left(\bigcup_{(C_s,M)\in \mathcal{S}_{n,i}}\left(C_s,M\right)\right) \mu\left( \bigcup_{a\in \mathcal{A}_{i,j}} C_{[a]}\right)\\
&= \mu\left( \bigcup_{a\in \mathcal{A}_{i,j}} C_{[a]}\right),
\end{align*}
with the same implicit constant $\mathcal{C}$ as in \eqref{eq:renyi}. Thus $K_{n_1} \asymp K_{n_2}$ for any $n_1,n_2\in \mathbb{N}$ with uniform implicit constant $\mathcal{C}^2$. This also implies that for any fixed $L$, we have that $K_1^L \asymp K_{n_1L,(n_1+1)L} \asymp K_{n_2L,(n_2+1)L}$ uniformly for any $n_1,n_2\in \mathbb{N}$ with implicit constant $\mathcal{C}^{2L}$.

Since $\widetilde{T}$ is assumed to be transitive, we know that for any $M_i,M_k\in \mathcal{M}$, there exists a cylinder $(C_s,M_i)$ with $|s|=\ell$ such that $T^\ell(C_s,M_i)=(\Omega,M_j)$. This implies that $(K_1^\ell)_{i,j}\asymp (K_{0,\ell})_{i,j}>0$. In other words, $K_1$ is an irreducible matrix. Also, we should be able to find integers $\ell_1, \ell_2,\dots, \ell_{|\mathcal{M}|}\in \mathbb{N}$ so that $(K_1^{\ell_i})_{i,i}>0$. Since all terms of $K_1$ are non-negative by construction, if $(K_1^{\ell_i})_{i,i}>0$ then we have   $(K_1^{m\ell_i})_{i,i}>0$ for any $m\in \mathbb{N}$. Thus, if we let $L=\operatorname{lcm}(\ell_1,\ell_2,\dots,\ell_{|\mathcal{M}|})$, then we have that $K_1^L$ is strictly positive along its diagonal. 

Suppose $K_1^L$ is itself irreducible. Then since $K_1^L$ has non-negative coefficients with a strictly positive diagonal, there is some power of it such that every coefficient is strictly positive (see equation (8.3.5) on page 672 of  \cite{Meyer}). We may therefore apply Lemma \ref{lem:markovchain} to the sequence of matrices $\{K_{nL, (n+1)L}\}_{n=0}^\infty$. So there exists $c'\in (0,1)$ and $n'_0$ such that $v_{nL,i}\in (c',1-c')$ for all $i$ and all $n\ge n'_0$. 

Suppose $K_1^L$ is not irreducible. By Theorem 3.4.5 in \cite{BR}, since $K_1$ itself is irreducible, there exists a permutation matrix $P$ such that 
\begin{equation}\label{eq:permuted}
PK_1^L P^T = \left( \begin{array}{cccc} C_1 & 0 & \cdots & 0 \\ 0 & C_2 & \cdots & 0\\ \vdots & \vdots & \ddots & \vdots \\ 0 & 0 & \cdots & C_r \end{array} \right),
\end{equation}
where the $C_j$ are irreducible matrices. In this case, we would apply Lemma \ref{lem:markovchain} to each $C_j$ and thus show  that there exist $c'_j\in (0,1)$ and $n'_j$ such that $v_{nL,i}\in (c'_j,1-c'_j)$ for $n \ge n'_j$ and for indexes $i$ corresponding to the matrix $C_j$ after undoing the permutation. By taking $c'=\min\{c'_j\}$ and $n'_0 =\max\{n'_j\}$, we get the same result as in the previous paragraph.

No column of $K_1$ consists of all $0$'s (otherwise there would be an $M\in \mathcal{M}$ that is never visited, contrary to the transitivity of $\widetilde{T}$), therefore the sum of the coefficients in any column vector of $K_{nL,nL+j}$ for $j\le L$ is uniformly bounded from below. Thus, we can therefore find a constant $c$ and $n_0$ such that $v_{n,i} \in (c,1-c)$ for all $i$ and all $n \ge n_0$. In particular, $v_{n,i}\asymp 1$.

Now we can prove the desired statement \eqref{eq:bigrenyi}. It suffices to show the statement is true for $n\ge n_0$ and for sets of the form $(E,M_i)$ for some measurable subset $E\subset\Omega$ and $M_i\in \mathcal{M}$. In this case, we have that $\widetilde{T}^{-n}(E,M_i)$ equals the union of $((T^{-n}E)\cap C_s, M)$ for $(C_s,M)\in \mathcal{S}_{n,i}$ as defined above. Therefore, by applying \eqref{eq:renyi}, we have
\begin{align*}
\tilde{\mu} \left( \widetilde{T}^{-n}(E,M_i) \right) &= \sum_{(C_s,M)\in \mathcal{S}_{n,i}} \tilde{\mu}\left( (T^{-n}E)\cap C_s, M\right) = \sum_{(C_s,M)\in \mathcal{S}_{n,i}} \frac{1}{|\mathcal{M}|} \mu((T^{-n}E)\cap C_s)\\
&\asymp \mu(E)  \sum_{(C_s,M)\in \mathcal{S}_{n,i}} \frac{1}{|\mathcal{M}|} \mu(C_s) = \mu(E) v_{n,i} \asymp \mu(E)\\
&= |\mathcal{M}| \cdot \tilde{\mu}(E,M_i) \asymp \tilde{\mu}(E,M_i),
\end{align*}
as desired.
\end{proof}

\subsection{Proof of Theorem \ref{thm:traversing}}
First, we will show that $\widetilde{T}$ is ergodic (despite not necessarily being invariant) with respect to $\tilde{\mu}$.

Suppose, we have a $\widetilde{T}$-invariant subset of $\widetilde{\Omega}$ called $E$ that has non-zero $\tilde{\mu}$-measure. We define $E^c= \widetilde{\Omega}\setminus E$. By projecting $E$ onto the first coordinate, we see that the set
\[
\left\{x\in \Omega \mid \text{There exists }M \in \mathcal{M}\text{, with }(x, M) \in E\right\}
\]
must have full $\mu$-measure, since it is invariant under $T$ and $T$ is ergodic with respect to $\mu$. 

 Let $E_{M}$ denote the set of $x\in \Omega$ such that $(x, M) \in E$. We claim that 
\[
\tilde{\mu}(E_{M} ,M)  >0
\]
for all $M\in \mathcal{M}$. To show this, let $M$ be fixed and let $s$ be a traversing string, whose existence is guaranteed by Lemma \ref{lem:traversing}.  Since $E$ must project onto a full $\mu$-measure set in $\Omega$ as described earlier, there must exist at least one $M' \in \mathcal{M}$ with 
\[
 \tilde{\mu}\left(\left(C_{s} ,M' \right)\cap E\right) > 0.
\]
By the definition of being a traversing string, however, there exists some $i<|s|$ such that 
\[
 (C_{s},M')\subset \widetilde{T}^{-i}\left(\Omega ,M\right),
\]
which implies that
\[
0< \tilde{\mu}\left(\left(C_{s} ,M' \right)\cap E\right)  \le \tilde{\mu}\left(\widetilde{T}^{-i}\left(\Omega ,M \right)\cap E\right) = \tilde{\mu}\left(\widetilde{T}^{-i}\left((\Omega ,M) \cap E\right)\right).
\]
But since $T$ is non-singular---that is, the preimage of a null set is itself a null set---so must $\widetilde{T}$ be. This can be seen by projecting onto the first coordinate again. Therefore, 
\[
0< \tilde{\mu}\left((\Omega ,M )\cap E\right)= \tilde{\mu} \left( E_{M} ,M \right),
\]
as desired.

Now we wish to show that $E$ has a substantial intersection with every cylinder set on $\widetilde{\Omega}$, in particular, by showing that there exists a constant $\epsilon>0$ so that for all cylinder sets $(C_s,M)$, we have
\begin{equation}\label{eq:desired1}
\tilde{\mu} (E \cap (C_s,M)) \ge \epsilon \cdot \widetilde{\mu}(C_s,M).
\end{equation}
(There is no relation between the strings $s$ considered from here on and the traversing string considered earlier.)

 Since there are only finitely many elements in $\mathcal{M}$, there must exist $\epsilon'>0$, such that
$
 \tilde{\mu}(E_{M'}) \ge \epsilon'
$
for all $M'\in \mathcal{M}$. Let us now fix an arbitrary cylinder $(C_s,M)$ with $n:=|s|$, and let $M'$ be such that $\widetilde{T}^{n}(C_s,M)=(\Omega,M')$. By applying \eqref{eq:renyi}, we have
\begin{align*}
\tilde{\mu}(E\cap (C_s,M)) &= \tilde{\mu}\left(E\cap \left( \widetilde{T}^{-n}(\Omega,M') \right) \cap (C_s,M)\right)\\
&= \tilde{\mu}\left( (C_s,M) \cap \widetilde{T}^{-n}\left( E \cap (\Omega,M')\right) \right)\\
&= \tilde{\mu}\left( (C_s,M) \cap \widetilde{T}^{-n}(E_{M'},M')\right)\\
&= \frac{1}{|\mathcal{M}|} \mu\left(T^{-n}E_{M'} \cap C_s \right)\ge \frac{1}{\mathcal{C}|\mathcal{M}|} \mu(E_{M'}) \mu(C_s)\\
&\ge \frac{\epsilon'}{\mathcal{C}|\mathcal{M}|} \mu(C_s)= \frac{\epsilon'}{\mathcal{C}} \tilde{\mu}(C_s,M).
\end{align*}
Therefore, letting $\epsilon = \epsilon'/\mathcal{C}$ gives \eqref{eq:desired1}.

Since the cylinder sets generate the Borel sets on $\widetilde{\Omega}$, we can find, for any $\delta>0$, a set $E_{\delta}$ such that $\tilde{\mu}(E^c \triangle E_{\delta})<\delta$ and $E_{\delta}$ is a disjoint union of a finite number of cylinder sets. Therefore, by applying \eqref{eq:desired1}, we have
\begin{align*}
\tilde{\mu}\left( E \cap E^c \right) &=\tilde{\mu}\left( E \cap E_{\delta} \right) + O( \delta) \ge \epsilon \tilde{\mu}(E) \tilde{\mu}(E_{\delta}) + O(\delta)\\
&= \epsilon \tilde{\mu}(E) \tilde{\mu}(E^c) + O (\delta).
\end{align*}
But $\tilde{\mu}(E \cap E^c ) = 0$ and $\delta$ was an arbitrary positive number. Thus either $\tilde{\mu}(E)=0$ or $\tilde{\mu}(E^c)=0$. Since we know $E$ has positive measure, this therefore implies that $E$ must have full measure, and $\widetilde{T}$ is ergodic with respect to $\tilde{\mu}$.

We will now construct a measure $\rho$ that is absolutely continuous with respect to $\tilde{\mu}$ such that $\widetilde{T}$ is not only ergodic but also invariant with respect to $\rho$.

We define a sequence of measures $\rho_n$ on $\widetilde{\Omega}$ by
\begin{equation}\label{eq:rhon}
\rho_n(A) = \frac{1}{n} \sum_{k=0}^{n-1} \tilde{\mu}\left( \widetilde{T}^{-k} A\right)= \int_{\widetilde{\Omega}} \left( \frac{1}{n} \sum_{k=0}^{n-1} \chi_A(T^{k} t)\right) d\tilde{\mu}(t).
\end{equation}
By Lemma \ref{lem:bigrenyi}, we can show that
\[
\limsup_{n\to \infty} \frac{1}{n} \sum_{k=0}^{n-1} \tilde{\mu}\left( \widetilde{T}^{-k} E \right) \le \mathcal{D} \tilde{\mu}(E),
\]
for any measurable set $E$. Therefore, by a theorem of Ryll-Nardzewski (see page 683 of \cite{DS}), the integrand of \eqref{eq:rhon} converges pointwise  to a $L_1$ function $g_A$ almost everywhere, and since the integrand is dominated by $1$, the integrand must in fact converge uniformly to $g_A$ almost everywhere. Therefore, we may define $\rho(A)= \lim_{n\to \infty} \rho_n(A)$. The Vitali-Hahn-Saks theorem \cite{VHS} shows that $\rho$ is in fact a probability measure on $\widetilde{\Omega}$. Since $\rho_n(\widetilde{T}^{-1} E) = \rho_n(E) +O(1/n)$ for any measurable set $E$, we have that $\rho$ is preserved by $\widetilde{T}$. Likewise, by Lemma \ref{lem:bigrenyi} again, we can see that
\begin{equation}\label{eq:mutualcont}
\frac{\tilde{\mu}(E)}{\mathcal{D}} \le \rho_n(E) \le \mathcal{D} \tilde{\mu}(E).
\end{equation}
and thus the same is true if we replace $\rho_n$ by $\rho$.\footnote{This part of the proof draws heavily on techniques used in the proof of Theorem 5.3.5 in \cite{IG}.}

Thus it remains to show that if $x\in\Omega$ is CF-normal, then $(x,M)$ is $\widetilde{T}$-normal with respect to $\rho$ for any $M\in\mathcal{M}$.

So consider a point $x\in \Omega$ that is CF-normal. Then for every cylinder $C_s$ and every $M\in \mathcal{M}$, we have
\begin{align*}
\lim_{N\to \infty} \frac{1}{N}\#\{ 1\le n \le N \mid \widetilde{T}^n (x, M) \in (C_s,\mathcal{M})\} &=\lim_{N\to \infty} \frac{1}{N}\#\{ 1\le n \le N \mid T^n x \in C_s \} \\
&= \mu(C_s) .
\end{align*}
Thus, in particular, we have for any $M,M' \in \mathcal{M}$ 
\begin{align*}
&\limsup_{N\to \infty} \frac{1}{N}\#\{ 1\le n \le N \mid \widetilde{T}^n (x, M) \in(C_s,M')\}\\
 &\qquad\le \limsup_{N\to \infty} \frac{1}{N}\#\{ 1\le n \le N \mid \widetilde{T}^n (x,M) \in (C_s,\mathcal{M})\}\\
&\qquad= \mu(C_s)= \frac{1}{|\mathcal{M}|} \tilde{\mu}(C_s,M')\le \frac{\mathcal{D}}{|\mathcal{M}|} \rho(C_s,M').
\end{align*}
Thus by Lemma \ref{lem:PS}, the points $(x, M)$ for \emph{all} $M\in \mathcal{M}$ are $\widetilde{T}$-normal with respect to $\rho$. \qed

\begin{rem}\label{rem:mutualcont}
We may consider Lebesgue measure on $\widetilde{\Omega}$ to be the Lebesgue measure on $\Omega$ crossed with the counting measure on $\mathcal{M}$. Since the Lebesgue measure and Gauss measure on $\Omega$ are absolutely continuous with respect to one another, and since we have \eqref{eq:mutualcont}, we see that Lebesgue measure (on $\widetilde{\Omega}$) and $\rho$ are absolutely continuous with respect to one another as well.
\end{rem}

\section{Building a dynamical system}\label{sec:buildingdynam}

As we hinted at the start of Section \ref{sec:tildespace}, we would like to build a dynamical system $\widetilde{T}$ from $\Omega\times \mathcal{M}_D$ to itself by $\widetilde{T}(x,M)=(Tx,U([a_1],M))$. We will use this definition for $\widetilde{T}$ throughout the rest of the paper; however, it may turn out that this system is not transitive and thus Theorem \ref{thm:traversing} may not apply. Thus we will require the following definition to modify our dynamical system.

We will call a subset $\mathcal{M}_D' \subset \mathcal{M}_D$ a transitive component if the following conditions are satisfied:
\begin{enumerate}
\item For any proper string $s$, and any $M\in \mathcal{M}_D'$ we have that $U(s,M)\in \mathcal{M}_D'$.
\item For any $M,M'\in \mathcal{M}_D'$ there exists a proper string $s$ such that $U(s,M)=M'$.
\end{enumerate}
Note that any two distinct transitive components of $\mathcal{M}_D$ must have empty intersection.

With this definition the transformation $\widetilde{T}$ given by $\widetilde{T}(x,M)=(Tx,U([a],M))$ \emph{is} transitive on $\Omega \times \mathcal{M}_D'$ for any transitive component $\mathcal{M}_D'\subset \mathcal{M}_D$.

\begin{lem}\label{lem:transitivecomponent}
There exists at least one transitive component of $\mathcal{M}_D$. Moreover, there exists a string $s$ such that $U(s,M)$ is in a transitive component for any $M\in \mathcal{M}_D$ (although not necessarily always in the same transitive component).
\end{lem}

\begin{proof}
Consider a directed graph $G$ whose vertices are matrices $M\in \mathcal{M}_D$ and which has an edge from $M_1$ to $M_2$ if there exists a $j\in \mathbb{N}$ such that $U([j],M_1)=M_2$. Note this graph has out-degree always at least $1$. A subgraph $G'$ of $G$ is said to be strongly connected if for any $M_1$, $M_2$ in $V(G')$, the vertex set of $G'$, there exists a path from $M_1$ to $M_2$ and vice-versa. We can partition $G$ into its strongly-connected components, which are the maximal strongly-connected subgraphs of $G$. (Note that if there is a vertex $M$ such that there is no path from $M$ to another vertex and back to itself, then $M$ is its own strongly-connected component.) Let us call these components $G_1, G_2, \dots, G_n$. Note that if $M\in V(G_i)$, then $V(G_i)$ consists of all vertices which are strongly-connected to $M$. 

Now let us consider another directed graph $\mathcal{G}$ whose vertices are $G_1,G_2,\dots,G_n$ and where there is an edge from $G_i$ to $G_j$ with $i\neq j$ if there exists an edge from some $M_i \in G_i$ to some $M_j\in G_j$ in the directed graph $G$. We do not let $\mathcal{G}$ contain an edge which goes from a vertex to itself. Note that if there is an edge from $G_i$ to $G_j$, then by the strong-connectivity of these components, there is a path from any $M_i\in G_i$ to any $M_j\in G_j$. We see that $\mathcal{G}$ cannot have any cycles, as otherwise it would be possible to find matrices in two different strongly-connected components that are strongly connected to one another, contradicting the maximality of these components. Thus $\mathcal{G}$ is acyclic.

Any finite acyclic directed graph must contain at least one sink. We claim that the set of vertices of any sink $G_i$ of $\mathcal{G}$ is a transitive component for $\mathcal{M}_D$. Let us fix a sink $G_i$ and let $\mathcal{M}'$ denote the matrices in $V(G_i)$. The first condition for being a transitive component is satisfied because if there existed a proper string $s$ and $M\in \mathcal{M}'$ such that $U(s,M)\not\in \mathcal{M}'$, then there would be at least one path from $G_i$ to another strongly-connected component, contradicting the assumption that $G_i$ is a sink. For the second condition, this follows from the fact that within any $G_i$ there is a path from any vertex to any other vertex and also to itself.

Now consider all the matrices that do not lie in a transitive component, let us call them $M_1, M_2,\dots, M_k$. Consider $M_1$ and suppose it is in $G_i$. As there must be a path from $G_i$ to a sink of $\mathcal{G}$, there exists a string $s_1$ such that $U(s_1,M_1)$ is in a transitive component. Now consider $U(s_1,M_2)$.  Regardless of what matrix in $\mathcal{M}_D$ the matrix $U(s_1,M_2)$ happens to be, there is, by the same argument, a string $s_2$ such that $U(s_2,U(s_1,M_2))=U(s_1s_2, M_2)$ is in a transitive component. (The string $s_2$ could be empty if $U(s_1,M_2)$ is already in a transitive component.) Likewise there is a string $s_3$ such that $U(s_1s_2s_3,M_3)$ is in a transitive component, and so on. The desired string $s$ is simply $s_1s_2s_3\dots s_k$.
\end{proof}

The second part of Lemma \ref{lem:transitivecomponent} allows us to reduce the cases of Theorem \ref{thm:main} yet further. By Remark \ref{rem:firstsimplification}, it suffices to prove Theorem \ref{thm:main} in the case where $x\in [0,1)$ is a CF-normal number and $M\in \mathcal{M}_D$ for some $D$. Let $s$ be a string satisfying Lemma \ref{lem:transitivecomponent}. If $x$ is CF-normal, its continued fraction expansion contains $s$. Thus there exists some $n$ such that $\widetilde{T}^n(x,M)\in \Omega\times \mathcal{M}'_D$, for some transitive component $\mathcal{M}'_D \subset \mathcal{M}_D$, and so
\[
Mx = R([a_1,a_2,\dots,a_n],M).\left( U([a_1,a_2,\dots,a_n],M) (T^n x )\right),
\]
with $U([a_1,a_2,\dots,a_n],M)\in \mathcal{M}_D'$. As noted in Remark \ref{rem:firstsimplification}, the action of $T$ and the action of strings both preserve CF-normality and CF-nonnormality, so it suffices to assume that $x\in[0,1)$ is CF-normal and that $M$ is in some transitive component $\mathcal{M}_D'$.

\begin{rem}
The dynamical system discussed above bears a non-trivial resemblance to the dynamical system on a skew-product studied by Fisher and Schmidt \cite{FS}, although there appears to be no direct overlap.
\end{rem}

\begin{lem}\label{lem:integrability}
 Let $\widetilde{T}$ and $\rho$ be the transformation and measure corresponding to some transitive component $\mathcal{M}_D'$, the latter of whose existence is gauranteed by Theorem \ref{thm:traversing}. 

Let $(C_{s_\ell},M_\ell)$ be a (possibly countable) sequence of cylinder sets  on $\widetilde{\Omega}=\Omega\times \mathcal{M}_D'$, such  that 
\begin{enumerate}
\item for any $x$ and $M$, we have $(x,M)\in (C_{s_\ell},M_\ell)$ for at most finitely many $\ell$; and,
\item for any $x$ and $M$, there exists a cylinder $(C_s,M)$ with $x\in (C_s,M)$ and such that no proper subcylinder $(C_{s'},M)\subset (C_s,M)$ is in the sequence $(C_{s_\ell},M_\ell)$.
\end{enumerate}
 Let $a_\ell>0$ be an associated sequence of positive real numbers so that the function
\[
f = \sum_{\ell} a_\ell \cdot 1_{(C_{s_\ell},M_\ell)} (\cdot)
\]
is bounded. Here $1_E(\cdot)$ is the standard indicator function of the set $E$.

Let $x\in[0,1)$ be CF-normal and $M\in \mathcal{M}_D$. Then
\[
\sum_{i=0}^{n-1} f(\widetilde{T}^i (x,M)) = n\cdot \left( \int_{\widetilde{\Omega}} f \ d\rho\right)(1+o(1)).
\]
\end{lem}

\begin{proof}
From the second part of Theorem \ref{thm:traversing}, we have that for any cylinder set $(C_s,M)$ that
\[
\sum_{i=0}^{n-1} 1_{(C_s,M)}(\widetilde{T}^i (x,M)) = n\cdot \left( \int_{\widetilde{\Omega}} 1_{(C_s,M)} \ d\rho\right)(1+o(1)).
\]
Thus, if the sequence $(C_{s_\ell},M_\ell)$ is finite, then the result follows immediately, so we suppose it is countably infinite instead.

We will refer to $D$ as an anti-cylinder if it equals the difference of a cylinder set with a finite number of other cylinder sets. For any anti-cylinder $D$, we can write the indicator function $1_D$ as a finite sum and difference of indicator functions of cylinder sets, thus we have that
\[
\sum_{i=0}^{n-1} 1_{D}(\widetilde{T}^i (x,M)) = n\cdot \left( \int_{\widetilde{\Omega}} 1_{D} \ d\rho\right)(1+o(1)).
\]

Now let us consider two sequences of functions $f_j^+$ and $f_j^-$. We let 
\[
f_j^- =\sum_{\ell\le j} a_\ell \cdot 1_{(C_{s_\ell},M_\ell)} (\cdot)
\]
so that $f_j^- \le f$ and $f_j^- \le f_{j+1}^-$. To define $f_j^+$ we first define the set $\mathcal{S}_j$, which will be a parition of $\Omega$. We let $\mathcal{S}_0=\{\Omega\}$. We then define all the $\mathcal{S}_j$ recursively, deriving $\mathcal{S}_j$ by letting it be the coarsest refinement of $\mathcal{S}_{j-1}$ that includes all cylinder sets $C_s$ with the sum of digits in $s$ equal to $j$. Thus 
\begin{align*}
\mathcal{S}_1 &= \{\Omega \setminus C_{[1]}, C_{[1]}\},\\
\mathcal{S}_2 &= \{\Omega \setminus (C_{[1]}\cup C_{[2]}), C_{[2]},C_{[1]}\setminus C_{[1,1]},C_{[1,1]}\},
\end{align*}
and so on. Note that $\mathcal{S}_j$ always consists of just cylinders and anti-cylinders. For any $S\in \mathcal{S}_j$, we let $a_{(S,M)}= \sup_{ (x,M)\in (S,M)} f(x,M)$ and then define
\[
f_j^+ = \sum_{S\in \mathcal{S}_j, M\in \mathcal{M}'_D} a_{(S,M)}\cdot 1_{(S,M)}(\cdot).
\]
Note that $f_j^+ \ge f$ and $f_j^+\ge f_{j+1}^+$.

We in fact have that $f_j^-$ and $f_j^+$ converge pointwise to $f$ (and hence in norm by dominated convergence). For $f_j^-$ this follows from assumption $(1)$ that any $(x,M)\in \widetilde{\Omega}$ is in finitely many of the cylinders $(C_{s_\ell},M_\ell)$, and hence for sufficiently large $j$, $f_j^-(x,M)=f(x,M)$. For $f_j^+$, this follows from assumption $(2)$, as for any $(x,M)\in \widetilde{\Omega}$ there exists a cylinder set $(C_s,M)$ containing $(x,M)$ such that $f$ is constant on $(C_s,M)$, thus for large $j$ we have $f_j^+(x,M)=f(x,M)$.

Thus we have, for any $j\ge 0$
\begin{align*}
\sum_{i=0}^{n-1} f_j^-(\widetilde{T}^i (x,M)) &\le \sum_{i=0}^{n-1} f(\widetilde{T}^i (x,M))  \le \sum_{i=0}^{n-1} f_j^+(\widetilde{T}^i (x,M)) \\
n \cdot \left( \int_{\widetilde{\Omega}} f_j^- \ d\rho\right) (1+o(1)) &\le \sum_{i=0}^{n-1} f(\widetilde{T}^i (x,M))  \le n \cdot \left( \int_{\widetilde{\Omega}} f_j^+ \ d\rho\right) (1+o(1)).
\end{align*}
As both integrals above converge in norm to $\int_{\widetilde{\Omega}} f \ d\rho$ as $j$ tends to infinity, this completes the proof.
\end{proof}

\section{Trigger strings}\label{sec:triggerstring}

Consider three strings $ s_+,s, s_-$ and three matrices $M_+, M, M_- $ in some transitive component $\mathcal{M}'_D\subset \mathcal{M}_D$ such that $U(s_+,M_+)=M$ and $U(s,M)=M_-$. With these definitions, note that \[\widetilde{T}^{|s_+|}(C_{s_+s}, M_+) = (C_s, M) \qquad\text{ and } \qquad \widetilde{T}^{|s|}(C_{ss_-},M) = C_{s_-},M_-).\] In this way, we can think of going from $(s,M)$ to $(s_+s,M_+)$ as prepending digits to the start of $(s,M)$, and likewise we can think of going from $(s,M)$ to $(ss_-,M)$ as appending digits to the end of $(s,M)$.

If we denote the resultant strings of $(s_+,M_+), (s,M), (s_-,M_-)$ by $t_+,t,t_-$ respectively with $t=[d_0;d_1,d_2,\dots,d_m]$, then it is easy to see that the resultant string of $(s_+s,M_+) $ should equal $t_+t$ and in fact, all the digits of $t$ appear unchanged and in the same position (relative to the end of the string) except perhaps $d_0$ and $d_1$. We say that these digits appear in the same \emph{relative position}. Likewise $R(ss_-,M) =tt_-$ and all the digits of $t$ appear unchanged and in the same position (now relative to the \emph{front} of the string) except perhaps $d_{m-1}$ and $d_m$. Again we say these digits appear in the same relative position.

We say a substring $r$ \emph{appears nicely} within the resultant string $t$ if it starts after the first digit and ends before the second-to-last digit. Thus, no matter what is prepended or appended to $s$, $r$ will still appear in the corresponding resultant string $t_+tt_-$. Although $r$ could appear many times in this longer string, there is exactly one copy of $r$ that occupies the \emph{relative position} corresponding to the original copy of $r$ in $t$.

We can run this procedure in reverse as well. If we start with a very long string $\overline{s}$ and matrix $\overline{M}\in \mathcal{M}$, we could decompose $\overline{s}$ as $s_+ss_-$ with corresponding matrices $M_+,M,M_-$. (In this case, we would refer to $(s,M)$ as a substring of $(\overline{s},\overline{M})$.) We could obtain other decompositions by removing the first digit of $s$ and appending it to the end of $s_+$, and this would alter the head of the resultant string $t$ (as well as the matrix $M$). Similarly we could remove the last digit of $s$ and prepend it to $s_-$, altering the tail of $t$.

Let $\overline{t}=R(\overline{s},\overline{M})$. Suppose we have a proper string $r$ that appears nicely in $\overline{t}$. While there are potentially many ways of decomposing $\overline{s}=s_+ss_-$ so that $r$ appears nicely in $t$ corresponding to the relative position of the original copy of $r$ in $\overline{t}$, there is a unique way of doing this so that the length of $s$ is minimized. Namely, we remove digits from the head or tail of $s$ and append them to $s_+$ or prepend tem to $s_-$ until removing any more digits would cause $r$ to no longer appear nicely within $t$ at the same relative position. We refer to this $(s,M)$ with the length of $s$ minimized as a trigger string for $r$. If there are exactly $k$ copies of $r$ in $\overline{t}$ all of which have the same minimal decomposition, we say that $(s,M)$ is a trigger string of multiplicity $k$.  For example if the resultant string of $([10],M)$ was $[1;1,1,1,1,1]$ and the desired string $r=[1]$, then this string has multiplicity $2$.

From these definitions, we see that the total number of times $r$ appears in $\overline{t}$ is equal (up to $O(1)$) to the number of trigger strings $(s,M)$ counted with multiplicity that occur in $(\overline{s},\overline{M})$. The $O(1)$ counts those four possible positions where $r$ could appear  in $\overline{t}$ but not nicely. 

\section{Proof of Theorem \ref{thm:main}}\label{sec:proof}

As we have noted in Remark \ref{rem:firstsimplification} and Section \ref{sec:buildingdynam}, it suffices to prove Theorem \ref{thm:main} in the case where $x\in [0,1)$ is CF-normal and $M\in \mathcal{M}_D'$ where $\mathcal{M}_D'$ is some transitive component of $\mathcal{M}_D$.

Note that Theorem \ref{thm:traversing} applies to $\widetilde{T}$ acting on the set $\Omega \times \mathcal{M}_D'$, giving us an ergodic, invariant measure $\rho$ on this space, and $(x,M)$ is normal with respect to $\widetilde{T}$ and $\rho$.

For the first and largest step of the proof, we want to show that for any finite proper string $r$, this string appears in $Mx$ with a limiting frequency that does not depend on $x$. (However, we will resume throughout the proof that $M$ is a fixed matrix.) In particular, we want a constant $\rho_r$ such that
\begin{equation}\label{eq:rhor}
\lim_{m\to \infty} \frac{\#\{0\le i \le m : T^i (Mx\  (\operatorname{mod }1)) \in C_r\}}{m} = \rho_r,
\end{equation}
for all CF-normal $x\in [0,1)$.

Let $x=\langle a_1,a_2,a_3,\dots\rangle$, and let $Mx =\langle b_0;b_1,b_2,\dots \rangle$ so that 
\[
Mx = R([a_1,a_2,\dots,a_n],M). \left( U([a_1,a_2,\dots,a_n],M) (T^n x)\right).
\]
We let $\ell(n)$ denote the length of $R([a_1,a_2,\dots,a_n],M)$. The following two lemmas assume that $r$ and $\mathcal{M}'_D$ are fixed and that $x$ is any CF-normal number in $[0,1)$ and $M\in \mathcal{M}'_D$. We will also let $\widetilde{\Omega}= \Omega \times \mathcal{M}_D'$.

\begin{lem}\label{lem:constantone}
We have for some constant $c_1>0$
\[
\ell(n)=c_1 n(1+o(1)).
\]
\end{lem}

\begin{proof}
We note that if the  resultant strings of $([a_i], U([a_1,a_2,\dots,a_{i-1}],M))$, for $i\le n$, all have a zeroth digit that is non-negative, then we clearly have that $\ell(n)$ is the sum of the lengths of these resultant strings. We may call this the ``expected" length. However, if $R([a_i], U([a_1,a_2,\dots,a_{i-1}],M))=[*,1]$ and $R([a_{i+1}], U([a_1,a_2,\dots,a_{i}],M))=[-1;*]$, then some of the digits cancel, removing two digits from the expected string length. This exception has its own exception: we must be careful if the resultant strings of  $([a_i], U([a_1,a_2,\dots,a_{i-1}],M))$ and $([a_{i+1}], U([a_1,a_2,\dots,a_{i}],M))$  are both $[-1;1]$, as these will essentially annihilate one another.

With this in mind, let us consider a set $\mathcal{S}$ consisting of pairs $(s,M)$ where $s$ is a string and $M$ is in our transitive component $\mathcal{M}'_D$, along with a function $g$ that acts on these pairs such that the following holds true.
\begin{enumerate}
\item All pairs $([c_1],M)$ whose resultant string is $[-1;1]$ are in $\mathcal{S}$; and $g([c_1],M)=0$ here.
\item All pairs $([c_1],M)$ where the last digit of the resultant string is greater than $1$ are in $\mathcal{S}$; and $g([c_1],M)$ equals the length of the resultant string.
\item All pairs $([c_1,c_2,\dots,c_j],M)$ where the last digit of $R([c_1],M)$ is $1$, \[ R([c_i],U([c_1,c_2,\dots,c_{i-1}],M))=[-1;1] ,\qquad 2\le i \le j-1,\] and $R([c_j],U([c_1,\dots,c_{j-1}],M))\neq [-1;1]$ are all in $\mathcal{S}$; and $g$ of this pair is the length of $R([c_1,c_2,\dots,c_j],M)$ minus the length of $R([c_j],U([c_1,\dots,c_{j-1}],M))$.
\end{enumerate}
By Lemma \ref{lem:contractingterms}, we see that if $(s,M)\in \mathcal{S}$, then $s$ has bounded length. Moreover, the cylinder sets $(C_s,M)$ for $(s,M)\in \mathcal{S}$ form a partition of $\widetilde{\Omega}$.

Now let us have a function $G$ on $\widetilde{\Omega}$ defined so that $G(x,M)=g(s,M)$ if $x\in C_s$ for some $s$ with $(s,M) \in \mathcal{S}$. By Lemmas \ref{lem:transitionlength} and \ref{lem:contractingterms}, we see that $G$ is uniformly bounded. Note that
\[
G = \sum_{(s,M)\in \mathcal{S}} g(s,M) \cdot 1_{(C_s,M)}(\cdot)
\]
and that $G(x,M)>0$ for some $(x,M)$.

From the above definitions we have that $\ell(n) = \sum_{i=0}^{n-1} G(\widetilde{T}(x,M)) +O(1)$. The $O(1)$ term accounts for small oddities that may occur near the beginning or end of the resultant string $R([a_1,a_2,\dots,a_n],M)$. The conditions of Lemma \ref{lem:integrability} hold (trivially so, as none of the cylinders in $\mathcal{S}$ overlap), and so $\ell(n)=n \cdot \left(\int_{\widetilde{\Omega}} G \ d{\rho} \right)\left(1+o(1)\right)$. Since $\rho$ is absolutely continuous with respect to the Lebesgue measure on $\widetilde{\Omega}$ and vice-versa by Remark \ref{rem:mutualcont}, this completes the proof.
\end{proof}

\begin{lem}\label{lem:constanttwo}
We have for some constant $c_r>0$,
\[
\#\{0 \le i \le \ell(n): T^i (Mx\ (\operatorname{mod }1)) \in C_r\}=c_r n(1+o(1)).
\]
\end{lem}

\begin{proof}
By Section \ref{sec:triggerstring}, we see that $\#\{0 \le i \le \ell(n): T^i (Mx\ (\operatorname{mod }1)) \in C_r\}$ equals the number of trigger strings for $r$ (counted with multiplicity) that are substrings of $[a_1,a_2,\dots,a_n]$, up to $O(1)$ to account for those $r$ that may not appear nicely within the resultant string. If the trigger strings had uniformly bounded length, this would be  equal to the number of times $\widetilde{T}^i (x,M)$ lands in a cylinder set $(C_s,M')$ corresponding to a trigger string $(s,M')$ counted with multipllicity. Again this would be up to $O(1)$ to account for $\widetilde{T}^i (x,M)$ landing in cylinder sets that make use of digits of $x$ beyond the $n$th digit. We do not know if the trigger strings have bounded length (although we strongly suspect this is the case), so we must be a bit more careful in our proof, using a sequence of functions $F_j^-$ and $F_j^+$ which will undercount or overcount the number of appearances of the trigger strings. 

As there are countably many strings of finite length, there are countably many trigger strings for $r$, so let us order them as $(s_i,M_i)$ in such a way so that if $(C_{s_{i'}},M_{i'})\subset (C_{s_i},M_{i})$ then $i'\ge i$.

Define a function $F_j^-$ on $\widetilde{\Omega}$ such that $F_j^-(x,M)$ counts the number of trigger strings $(s_i,M_i)$ (with multiplicity) with $i\le j$ such that $x\in C_{s_i}$ and $M=M_i$. If we let $k_i$ denote the multiplicity of $(C_{s_i},M_i)$, then we have
\[
F_j^- = \sum_{i\le j} k_i \cdot 1_{(C_{s_i},M_i)}.
\]

From our discussion in the first paragraph of the proof, we see that \[
\sum_{i\le n} F_j^- \left(\widetilde{T}^i (x,M)\right) +O_j(1)\le
\#\{0 \le i \le \ell(n): T^i (Mx) \in C_r\}.
\] 
The function $F_j^-$ clearly satisfies the conditions for Lemma \ref{lem:integrability} as it is a finite sum, so
\begin{equation}\label{eq:undercount}
n\cdot\left( \int_{\widetilde{\Omega}} F_j^- \ d\rho\right) (1+o(1))\le
\#\{0 \le i \le \ell(n): T^i (Mx) \in C_r\}
\end{equation}
for fixed $j$.

If $(s,M)$ is a trigger string for $r$ and the first digit of $s$ is $c$, then the first digit for this appearance of $r$ must be within the first $k+2$ digits of $R(s,M)$, where $k$ is the length of $R([c],M)$; otherwise we could remove $c$ from the start of $s$, and $r$ would still appear nicely within the shortened resultant string in the same relative position, contradicting the minimality condition for trigger strings.  It is possible that if $(s_i,M_i)$ is a trigger string for $r$, then there could be a string $s_-$ so that $(s_is_-,M_i)$ is also a trigger string for $r$, but these two appearances of $r$ cannot start at the same place (relative to the start of the resultant string). Therefore, if we let 
\[
K=2+\max_{c\in \mathbb{N},M'\in \mathcal{M}'_D} \left| R([c],M') \right|,
\]
and by Lemma \ref{lem:transitionlength}, the maximum here really does exist, then any given point in $\widetilde{\Omega}$ can lie in at most $K$ cylinders $(C_{s_i},M_i)$  corresponding to trigger strings $(s_i,M_i)$ (counted, as usual, with multiplicity).

We then define another function $F_j^+$ on $\widetilde{\Omega}$ by 
\[
F_j^+=\begin{cases}
K, &\text{if }(x,M)\in (C_{s_i},M_i)\text{ for some }i>j,\\
F_j^-(x,M), & \text{otherwise.}
\end{cases}
\]
Now we have that $ \#\{0 \le i \le \ell(n): T^i (Mx) \in C_r\}$ is less than the number of times $\widetilde{T}^i(x,M)$ is in a cylinder set corresponding to a trigger string (counted with multiplicity) and this in turn is less than 
\[
 \sum_{i\le n} F_j^+ \left(\widetilde{T}^i (x,M)\right) 
\]

We can write $F_j^+$ another way, as
\[
F_j^+=F_j^- + \sideset{}{^*}\sum_{i>j} K_i \cdot 1_{(C_{s_i},M_i)}
\]
where the starred sum runs over all $i>j$ such that $(C_{s_i},M_i)$ is not a subset of any other $(C_{s_{i'}},M_{i'})$ with $i'>j$, and $K_i$ is equal to $K$ minus the value of $F_j^- $ on this cylinder set. These $K_i$ are well-defined because of the ordering we chose for the trigger strings.  Since the starred sum is over a disjoint union of cylinder sets, the function $F_j^+$ also satisfies the conditions for Lemma \ref{lem:integrability}. Therefore 
\begin{equation}\label{eq:overcount}
\#\{0 \le i \le \ell(n): T^i (Mx) \in C_r\} \le n \cdot \left( \int_{\widetilde{\Omega}} F_j^+ \ d\rho\right)(1+o(1))
\end{equation}
for fixed $j$.

Since any point in $\widetilde{\Omega}$ appears in a uniformly bounded number of cylinder sets, the functions $F_{j}^-$ and $F_{j}^+$ converge pointwise as $j$ tends to infinity. And since both functions are uniformly bounded by the constant function $K$, they must converge in norm as well by the dominated convergence theorem. Combining this with \eqref{eq:undercount} and \eqref{eq:overcount} gives the desired result. 
\end{proof}

From here we are nearly done. First, let $\ell^{-1}(m):= \max\{n: \ell(n)\le m\}$. By Lemma \ref{lem:constantone}, $\ell^{-1}(m) = m/c (1+o(1))$, so therefore we have that \[ 0\le m - \ell(\ell^{-1}(m)) \le o(m)\]

Thus we have that
\[
\lim_{m\to \infty}  \frac{\#\{0 \le i \le m: T^i (Nx) \in C_r\}}{m} = \lim_{m\to \infty} \frac{\#\{0 \le i \le \ell(\ell^{-1}(m)): T^i (Nx) \in C_r\}+o(m)}{\ell(\ell^{-1}(m))+o(m)} 
\]
and so \eqref{eq:rhor} follows from Lemma \ref{lem:constantone} and \ref{lem:constanttwo}.

Now consider the sets
\[
E_M = \{Mx: x\in [0,1)\text{ is CF-normal}\} \text{ and } E = \{x: x\in \mathbb{R} \text{ is CF-normal}\}.
\]
We have shown that for any string $r$ there exists a constant $\rho_r$ such that for all $y\in E_M$, the string $r$ appears in the continued fraction expansion of $y$ with limiting frequency $\rho_r$, even though we do not know what any of these constants $\rho_r$ equal. On the other hand, for all strings $r$ and all $x\in E$, the limiting frequency of $r$ in the continued fraction expansion of $x$ is $\mu(C_r)$. Thus, either $\rho_r=\mu(C_r)$ for all $r$ and $E_M$ is a subset of $E$ or $\rho_r\neq \mu(C_r)$ for some $r$ and $E_M$ is disjoint from $E$. 

However, $E$ has full Lebesgue measure and $E_M$, being a non-trivial linear fractional transformation of a positive measure set, has positive measure, so $E_M$ must be a subset of $E$, and the theorem is proved.

\section{Further questions}

In one---admittedly peculiar---sense, the generalization that we have proved of Wall's result is not the natural generalization to make. What makes rational numbers so nice for any base $b$, is that they are eventually periodic. So one could ask the following.

Suppose $x$ is CF-normal and $q$ and $r$ have eventually periodic continued fraction expan\-sions---that is, they are both quadratic irrationals---with $q\neq 0$. Must it be true that $qx+r$ is CF-normal as well?

Also, while Theorem \ref{thm:main} solves Bugeaud's problem, it does not solve Mend\`{e}s France's problem. CF-normality is a much stronger condition than CF-simple normality, and our proof relies crucially on full CF-normality. So we ask, as Mend\`{e}s France did: does non-zero rational multiplication and rational addition preserve CF-simple normality?

\section{Acknowledgments}

The author would like to thank Justin Moore for asking a thought-provoking question on \url{mathoverflow.net} regarding the effect adding $1/2$ has on a continued fraction expansion, Bill Mance for bringing Bugeaud's question to his attention, and Cor Kraaikamp for pointing the author to the work of Liardet and Stambul.

The author acknowledges assistance from the Research and Training Group grant DMS-1344994 funded by the National Science Foundation.

\bibliographystyle{amsplain}

\end{document}